\documentclass[a4paper,10pt]{amsart}
\usepackage[utf8]{inputenc}
\usepackage{amsxtra}
\usepackage{amsopn}
\usepackage{amsmath,amsthm,amssymb}
\usepackage{amscd}
\usepackage{mathtools}
\usepackage{amsfonts}
\usepackage{latexsym}
\usepackage{verbatim}
\usepackage{xcolor}
\newcommand{\lr}{\longrightarrow}

\newcommand{\la}{\ll}
\newcommand{\ra}{\gg}
\theoremstyle{plain}
\newtheorem{theorem}{Theorem}[section]
\newtheorem*{theorem*}{Theorem}

\newtheorem{lemma}[theorem]{Lemma}
\newtheorem{proposition}[theorem]{Proposition}

\newtheorem{example}[theorem]{Example}
\newtheorem*{mt*}{Main Theorem}
\newcommand\C{{\mathbb C}}

\renewcommand\phi{{\varphi}}

\newcommand\N{{\mathbb N}}
\newcommand\R{{\mathbb R}}

\newcommand{\del}{\partial}
\newcommand{\delbar}{\overline{\del}}
\newcommand\SL{{\hbox{\em SL}}}
\renewcommand\sl{\mathfrak{sl}}
\newcommand\g{\mathfrak{g}}
\newcommand{\Cinf}{$\mathcal{C}^\infty$}

\DeclareMathOperator{\supp}{supp}

\let\sup\undefined
\DeclareMathOperator*{\sup}{sup\vphantom{p}}
\DeclareMathOperator{\vol}{Vol}

\DeclareMathOperator{\real}{Re}

\DeclareMathOperator{\im}{im}
\let\span\undefined
\DeclareMathOperator{\span}{Span}

\title{Aeppli Cohomology and Gauduchon Metrics}
\author{Riccardo Piovani}
\address{Dipartimento di Matematica\\
Universit\`{a} di Pisa\\
Largo Bruno Pontecorvo 5 \\
56127 Pisa, Italy}
\email{piovani@mail.dm.unipi.it}

\author{Adriano Tomassini}
\address{Dipartimento di Scienze Matematiche, Fisiche e Informatiche\\
Unit\`{a} di Matematica e Informatica\\
Universit\`{a} degli Studi di Parma\\
Parco Area delle Scienze 53/A \\
43124 Parma, Italy}
\email{adriano.tomassini@unipr.it}

\thanks{This work was partially supported by the Project PRIN ``Varietà reali e complesse: geometria, topologia e analisi armonica'' 
and by GNSAGA of INdAM}

\begin{document}
\maketitle

\begin{abstract}
Let $(M,J,g,\omega)$ be a complete Hermitian manifold of complex dimension $n\ge2$. Let $1\le p\le n-1$ and assume that $\omega^{n-p}$ is $(\partial+\overline{\partial})$-bounded. We prove that, if 
$\psi$ is an $L^2$ and $d$-closed $(p,0)$-form on $M$, then $\psi=0$. In particular, if $M$ is compact, we derive that if the Aeppli class of $\omega^{n-p}$ vanishes, then $H^{p,0}_{BC}(M)=0$. As a special case, if $M$ admits a Gauduchon metric $\omega$ such that the Aeppli class of $\omega^{n-1}$ vanishes, then $H^{1,0}_{BC}(M)=0$.
\end{abstract}

\section{Introduction}
Let $M$ be a compact smooth manifold of even dimension $2n$. In this note, we will consider only manifolds without boundary. A simple obstruction in order that $M$ carries a K\"ahler metric is that $H^{2k}_{dR}(M;\R)\ne0$ for every $1\le k\le n$. Indeed, if $(J,g,\omega)$ is a K\"ahler stucture on $M$, then it is well known that $\omega^k$ gives rise to a non-zero de Rham cohomology class on $M$, for every $1\le k\le n$. Furthermore, Hodge decomposition theorem on compact K\"ahler manifolds guarantees that complex de Rham cohomology decompose as the direct sum of Dolbeault cohomology groups and that these coincide with 
$(p,q)$-{\em Bott-Chern} cohomology and $(p,q)$-{\em Aeppli cohomology} groups, defined respectively as
\begin{equation*}
H^{*,*}_{BC}(M)=\frac{\ker\del\cap\ker\delbar}{\im\del\delbar},\ \ \ 
H^{*,*}_{A}(M)=\frac{\ker\del\delbar}{\im\del+\im\delbar}.
\end{equation*}

For a compact non K\"ahler manifold, things are more complicated. For example, there are compact complex manifolds endowed with balanced metrics, namely Hermitian metrics $g$ whose fundamental form $\omega$ satisfies $d\omega^{n-1}=0$, such that $0=[\omega^{n-1}]_{dR}\in H^{2n-2}_{dR}(M;\C)$ (see, e.g., \cite{OUV} or Example \ref{ex-sl2c}). Moreover, Hodge decomposition does not hold; therefore Bott-Chern and Aeppli cohomology groups are natural objects to study. A remarkable result by Gauduchon, \cite[Theorem 1]{Ga}, states that if $(M,J)$ is a compact almost complex manifold of real dimension $2n>2$, given any Hermitian metric $g$, with fundamental form $\omega$, there exists a unique Hermitian metric $\tilde{g}$, conformally equivalent to $g$, whose fundamental form $\tilde{\omega}$ satisfies $dd^c\tilde{\omega}^{n-1}=0$, where $d^c=J^{-1}dJ$. In particular, if $J$ is integrable, then $\tilde{\omega}$ satisfies $\del\delbar\tilde{\omega}^{n-1}=0$. Hence $\tilde{\omega}^{n-1}$ gives rise to a cohomology class in $H^{n-1,n-1}_A(M)$. A Hermitian metric $g$ on a compact complex $n$-dimensional manifold, whose fundamental form $\omega$ satisfies $\del\delbar\omega^{n-1}=0$, is said to be a {\em Gauduchon metric}.

In this note, given a Gauduchon metric $g$, with fundamental form $\omega$, on a compact complex manifold $M$ of complex dimension $n$, we are interested in studying whether the Aeppli cohomology class $[\omega^{n-1}]_{A}\in H^{n-1,n-1}_A(M)$ vanishes.
If $n=2$, i.e., on compact complex surfaces, an application of the Stokes Theorem shows that any Gauduchon metric gives rise to a non-zero class in Aeppli cohomology, i.e., $0\ne[\omega]_{A}\in H^{1,1}_A(M)$ (see, e.g., \cite[Prop. 37]{HL}).
Very recently Yau, Zhao and Zheng prove that if $(M,J,g,\omega)$ is a compact SKL (Strominger K\"ahler-like) non-K\"ahler manifold, then $\omega$ gives rise to a non-zero class in Aeppli cohomology, i.e., $0\ne[\omega]_{A}\in H^{1,1}_A(M)$ (see \cite[Thm. 1]{YZZ}).
In this paper, we prove the following result (see Theorem \ref{cor-cpt}).
\begin{theorem} \label{main-cpt}
Let $(M,J,g,\omega)$ be a compact Hermitian manifold of complex dimension $n\ge3$.  Let $1\le p\le n-1$. If $0=[\omega^{n-p}]_{A}\in H^{n-p,n-p}_A(M)$, then $H^{p,0}_{BC}(M)=0$.
\end{theorem}
As a consequence, we obtain the following (see Theorem \ref{gau}).
\begin{theorem}
Let $M$ be a compact complex manifold of complex dimension $n\ge3$. If $M$ carries a Gauduchon metric $\omega$ such that $0=[\omega^{n-1}]_{A}\in H^{n-1,n-1}_A(M)$, then $H^{1,0}_{BC}(M)=0$.
\end{theorem}
As an application of Theorem \ref{main-cpt}, following \cite{OUV}, we describe the example of $M=\Gamma\backslash G$, where $G$ is the complex Lie group $\SL(2,\C)$, and $\Gamma$ is a discrete cocompact subgroup (see Example \ref{ex-sl2c}). On $M$ we consider an invariant Hermitian metric $\omega$ such that  $\omega^2$ is exact. Therefore $0=[\omega^{2}]_{A}\in H^{2,2}_A(M)$ and, consequently, Theorem \ref{cor-cpt} applies, giving $H^{1,0}_{BC}(M)=0$.

Theorem \ref{main-cpt} derives from a more general result on complete Hermitian manifolds. Indeed, inspired by a vanishing theorem by Gromov, \cite[Theorem 1.2.B]{G}, we study the problem presented above in a more general setting, namely on complete Hermitian manifolds. In this context, Gromov \cite{G} got 
a vanishing result for $L^2$-Dolbeault cohomology, that is, he proved that if $(M,J,g,\omega)$ is a complete K\"ahler manifold of complex dimension $n$ such that $\omega$ is $d$-bounded, i.e., $\omega=d\eta$, with $\eta$ bounded, then $H^{p,q}_{dR}(M)=0$ for $p+q\ne n$.
In his proof, he made use of the $L^2$-Hodge decomposition theorem, and showed the vanishing of $L^2$ de Rham harmonic forms. In \cite{PT1} and \cite{PT2} (see also \cite{HT} for the almost K\"ahler setting), the authors of the present note extend Gromov's result for $W^{1,2}$ Bott-Chern harmonic forms, giving a characterization of $W^{1,2}$ Bott-Chern harmonic forms on Stein $d$-bounded manifolds, respectively on complete Hermitian manifolds. Along the same line, we prove the following result (see Theorem \ref{vanishing}), which implies Theorem \ref{main-cpt}.
\begin{theorem}
\label{main}
Let $(M,J,g,\omega)$ be a complete Hermitian manifold of complex dimension $n\ge2$. Let $1\le p\le n-1$ and assume that 
$\omega^{n-p}$ is $(\partial+\overline{\partial})$-bounded.
Let $\psi$ be a $(p,0)$-form on $M$ such that
$$
\psi\in L^2(M),\qquad \partial\psi=0,\qquad \overline{\partial}\psi=0.
$$
Then $\psi=0$.
\end{theorem}

The paper is organized in the following way. In section \ref{preliminaries}, we set notation and introduce the objects which will be studied in the following. In section \ref{main results}, we give two basic lemmas and prove the main Theorem \ref{vanishing}.
In section \ref{examples}, we give some examples, on compact quotients of $SL(2,\C)$, on the Calabi-Eckmann manifold $\mathbb{S}^3\times\mathbb{S}^3$ and on the Kodaira surface of secondary type. Finally we give two applications of Theorem \ref{main-cpt} to compact nilmanifolds endowed with an invariant complex structure (see Propositions \ref{prop1} and \ref{prop2}).

\medskip\medskip
\noindent{\em Acknowledgments.} We are grateful to Daniele Angella, Paul Gauduchon, and Valentino Tosatti for useful conversations and helpful comments.

\section{Preliminaries}\label{preliminaries}
\label{preliminaries}
Let $(M,J,g,\omega)$ be a Hermitian manifold of complex dimension $n$, where $M$ is a smooth manifold of real dimension $2n$, $J$ is 
a complex structure on $M$, $g$ is a $J$-invariant Riemannian metric on $M$, and $\omega$ denotes the fundamental $(1,1)$-form associated to the metric $g$. We denote by $h$ the Hermitian extension of $g$ on the complexified tangent bundle $T^\C M=TM\otimes_\R\C$, and by the same symbol $g$ the $\C$-bilinear extension of $g$ on $T^\C M$. Also denote by the same symbol $\omega$ the $\C$-bilinear extension of the fundamental form $\omega$ of $g$ on $T^\C M$. Recall that $h(u,v)=g(u,\bar{v})$ for all $u,v\in T^{1,0}M$, and $\omega(u,v)=g(Ju,v)$ for all $u,v\in TM$. The Hermitian metric $g$ is said to be {\em Gauduchon} if $\del\delbar\omega^{n-1}=0$.

We denote by $\Omega^r(M,\C)=\Gamma(\Lambda^rM\otimes_\R\C)$ the space of complex $r$-forms, and by $\Omega^{p,q}(M)=\Gamma(\Lambda^{p,q}M)$ the space of $(p,q)$-forms on $M$. Denoting by $*:\Omega^{p,q}(M)\to \Omega^{n-p,n-q}(M)$ the complex anti-linear Hodge operator 
associated with $g$, the {\em Bott-Chern Laplacian} and {\em Aeppli Laplacian} 
$\tilde\Delta_{BC}$ and $ \tilde\Delta_{A}$ are the $4$-th order elliptic differential operators defined respectively as (see \cite[p. 71]{KS} and \cite[p. 8]{S})
\begin{equation*}
\tilde\Delta_{BC} \;:=\;
\del\delbar\delbar^*\del^*+
\delbar^*\del^*\del\delbar+\delbar^*\del\del^*\delbar+
\del^*\delbar\delbar^*\del+\delbar^*\delbar+\del^*\del
\end{equation*}
and
\begin{equation*}
 \tilde\Delta_{A} \;:=\; \del\del^*+\delbar\delbar^*+
\delbar^*\del^*\del\delbar+\del\delbar\delbar^*\del^*+
\del\delbar^*\delbar\del^*+\delbar\del^*\del\delbar^*\,,
\end{equation*}
where, as usual
\begin{equation*}
\del^*:=-*\del\, *\,,\qquad \delbar^*:=-*\delbar\, *.
\end{equation*}
Note that $*\tilde\Delta_{BC}=\tilde\Delta_{A}*$ and $\tilde\Delta_{BC}*=*\tilde\Delta_{A}$, then $u\in\ker\tilde\Delta_{BC}$ if and only if $*u\in\ker\tilde\Delta_{A}$.
If $M$ is compact, then 
\begin{equation*}
u\in\ker\tilde\Delta_{BC}\iff\del u=0,\ \delbar u=0,\ \del^*\delbar^*u=0,
\end{equation*}
and
\begin{equation*}
v\in\ker\tilde\Delta_{A}\iff\del^* v=0,\ \delbar^* v=0,\ \del\delbar v=0.
\end{equation*}
Moreover, according to \cite{S}, $\ker\tilde\Delta_{BC}$ and $\ker\tilde\Delta_{A}$ are finite dimensional complex vector spaces, and
\begin{equation*}
\ker\tilde\Delta_{BC|_{\Omega^{p,q}(M)}}\cong H^{p,q}_{BC}(M),\ \ \ \ker\tilde\Delta_{A|_{\Omega^{p,q}(M)}}\cong H^{p,q}_{A}(M).
\end{equation*}
More explicitly, given any $(p,q)$-form $\phi$, we may decompose $\phi$ as
\begin{equation}\label{bc-decomp}
\phi=h_{BC}(\phi)+\del\delbar\gamma+\del^*\alpha+\delbar^*\beta,
\end{equation}
where $h_{BC}(\phi)\in H^{p,q}_{BC}(M)$, and
\begin{equation}\label{a-decomp}
\phi=h_{A}(\phi)+\del^*\delbar^*\eta+\del\mu+\delbar\lambda,
\end{equation}
where $h_{A}(\phi)\in H^{p,q}_{A}(M)$. We will refer to (\ref{bc-decomp}), respectively (\ref{a-decomp}), as the Bott-Chern, respectively Aeppli, decomposition of the $(p,q)$-form $\phi$.

Let $(M,J,g,\omega)$ be a Hermitian manifold of complex dimension $n$. Denote 
by $\vol=\frac{\omega^n}{n!}$ the standard volume form. 
Let $\langle\,,\rangle$ be the pointwise Hermitian inner product induced by $g$ on 
the space of tensors of complex bigrade $(p,q)$. 
Given any tensor $\phi$, set
\begin{gather*}
\vert\varphi\vert^2:=\langle\varphi,\varphi\rangle,\\
\lVert \varphi\rVert_{L^2}^2:=\int_M\vert\varphi\vert^2\vol,
\end{gather*}
and
\begin{equation*}
L^2(M):=\left\{\varphi\in\Omega^r(M)\,\,\,\vert\,\,\, 0\le r\le 2n,\ \lVert \varphi\rVert_{L^2}<\infty \right\}.
\end{equation*}
For any tensors $\varphi,\psi$, denote by $\la\,,\ra$ the $L^2$ Hermitian product defined by
\begin{equation*}
\la\varphi,\psi\ra :=\int_M\langle\varphi,\psi\rangle \vol.
\end{equation*}
For any given  tensor $\varphi$, we also set 
\begin{equation*}
\lVert \varphi\rVert_{L^\infty}:=\sup_{x\in M}\vert\varphi\vert(x)
\end{equation*}
and we call $\varphi$ {\em bounded} if $
\lVert \varphi\rVert_{L^\infty}<\infty$. Furthermore, if $\phi\in\Omega^r(M)$, $\varphi=d\eta$, and $\eta$ is bounded, then $\varphi$ is said to be {\em d-bounded}. 
We say that a $(p,q)$-form $\eta\in\Omega^{p,q}(M)$ is $(\del+\delbar)${\em-bounded} if $\eta=\del\mu+\delbar\lambda$, and $\mu$ and $\lambda$ are bounded. In particular $\del\delbar\eta=0$.

\section{Main results}\label{main results}
We start by proving the following lemmas.
\begin{lemma} \label{star}
Let $(M,J,g,\omega)$ be a Hermitian manifold of complex dimension $n\ge2$. Let $1\le p\le n-1$ and $\psi$ be a $(p,0)$-form on $M$. Set 
$$
c_{n,p}=(-1)^{\frac{p(p+1)}{2}}(-i)^{n-p}(n-p)!
$$
Then,
$$
*(\omega^{n-1}\wedge\psi)=c_{n,p}\overline{\psi}
$$
\end{lemma}
\begin{proof}
Let $\{\phi^1,\ldots,\phi^n\}$ be a local unitary coframe on $M$. Firstly, assume that $\psi$ is a $(p,0)$-form. Then, for $i_1<\dots<i_p$, $i_k=1,\dots,n$,
$$
\psi=\sum_{i_1<\dots<i_p} a_{i_1\dots i_p}\phi^{i_1\dots i_p},\quad \omega=\frac{i}{2}\sum_{j=1}^n \phi^j\wedge\overline{\phi^j},
$$
and
$$
\omega^{n-p}=(\frac{i}{2})^{n-p}(n-p)!
\sum_{i_1<\dots<i_p}\phi^{1\bar{1}\ldots\widehat{i_1\bar{i_1}}\ldots\widehat{i_p\bar{i_p}}\ldots n\bar{n}},
$$
where $\phi^{r\bar{s}}=\phi^r\wedge\overline{\phi^{s}}$ and $\widehat{s\bar{s}}$ means that the pair 
$s\bar{s}$ is omitted. Therefore,
$$
\omega^{n-p}\wedge\psi=(-1)^{\frac{p(p+1)}{2}}(\frac{i}{2})^{n-p}(n-p)!
\sum_{i_1<\dots<i_p}a_{i_1\dots i_p}\phi^{1\bar{1}\ldots i_1\hat{\bar{i_1}}\ldots i_p\hat{\bar{i_p}}\ldots n\bar{n}},
$$
and consequently
\begin{equation*}
\begin{split}
*(\omega^{n-p}\wedge\psi)&=2^{n-p}\overline{(-1)^{\frac{p(p+1)}{2}}(\frac{i}{2})^{n-p}(n-1)!}
\sum_{i_1<\dots<i_p}\overline{a}_{i_1\dots i_p}\phi^{\bar{i_1}\dots \bar{i_p}}\\
&=(-1)^{\frac{p(p+1)}{2}}(-i)^{n-p}(n-p)!\overline{\psi},
\end{split}
\end{equation*}
that is 
\begin{equation*}
*(\omega^{n-p}\wedge\psi)=c_{n,p}\overline{\psi}.\qedhere
\end{equation*}
\end{proof}

\begin{lemma}
\label{delstar}
Let $(M,J,g,\omega)$ be a Hermitian manifold of complex dimension $n\ge2$. Let $1\le p\le n-1$ and $\psi$ be a $(p,0)$-form on $M$ such that 
$$\partial\psi=0,\quad 
\overline{\partial}\psi=0.
$$
Then, 
$$
\partial^*(\omega^{n-p}\wedge\psi)=0,\qquad \overline{\partial}^*(\omega^{n-p}\wedge\psi)=0.
$$
\end{lemma}
\begin{proof}
Let $\psi$ be a closed $(p,0)$-form. By the definition of $\partial^*, \overline{\partial}^*$,
and Lemma \ref{star}, we have
\begin{equation*}
\begin{split}
\partial^*(\omega^{n-p}\wedge\psi)&= -*\partial*(\omega^{n-p}\wedge\psi)=-\overline{c_{n,p}}*\partial\overline{\psi}=0\\
\overline{\partial}^*(\omega^{n-p}\wedge\psi)&= -*\overline{\partial}*(\omega^{n-p}\wedge\psi)=-\overline{c_{n,p}}*\overline{\partial}\overline{\psi}=0.\qedhere
\end{split}
\end{equation*}
\end{proof}
We can now prove the following vanishing result.
\begin{theorem}
\label{vanishing}
Let $(M,J,g,\omega)$ be a complete Hermitian manifold of complex dimension $n\ge2$. Let $1\le p\le n-1$ and assume that 
$\omega^{n-p}$ is $(\partial+\overline{\partial})$-bounded.
Let $\psi$ be a $(p,0)$-form on $M$ such that
$$
\psi\in L^2(M),\qquad \partial\psi=0,\qquad \overline{\partial}\psi=0.
$$
Then $\psi=0$.
\end{theorem}
\begin{proof}
By completeness, as stated in \cite[Chapter VIII, Lemma 2.4]{DE}, let $\{K_\nu\}_{\nu\in\N}$ be a sequence of compact subsets of $M$, and $\{a_\nu:K_\nu\lr [0,1]\subset\R\}_{\nu\in\N}$ be a sequence of smooth cut off functions with compact support, such that the following properties hold:
\begin{itemize}
\item $\bigcup_{\nu\in\N} K_\nu=M$ and $K_\nu\subset \mathop K\limits^ \circ$$_{\nu+1}$;
\item $a_\nu=1$ in a neighbourhood of $K_\nu$ 
and $\supp{a_\nu}\subset \mathop K\limits^ \circ$$_{\nu+1}$;
\item $|d a_\nu|\le 2^{-\nu}$.
\end{itemize}

Let $\psi$ be as in the statement of the theorem.  Since $\del\psi=0$, $\delbar\psi=0$ and $\omega^{n-p}$ is $(\del+\delbar)$-bounded, note that
\begin{equation}\label{omegapsi}
\omega^{n-p}\wedge\psi=(\del\mu+\delbar\lambda)\wedge\psi=\del(\mu\wedge\psi)+\delbar(\lambda\wedge\psi).
\end{equation}
By Lemma \ref{delstar}, we have
\begin{equation}\label{del*}
\begin{split}
0&=\la\del^*(\omega^{n-p}\wedge\psi),a_\nu\mu\wedge\psi\ra\\
&=\la\omega^{n-p}\wedge\psi,\del(a_\nu\mu\wedge\psi)\ra\\
&=\la\omega^{n-p}\wedge\psi,\del a_\nu\wedge\mu\wedge\psi\ra+\la\omega^{n-p}\wedge\psi,a_\nu\del(\mu\wedge\psi)\ra.
\end{split}
\end{equation}
Then, since $\omega^{n-p}$ and $\mu$ are bounded, $\psi$ is $L^2$ and $|d a_\nu|\le 2^{-\nu}$, we get that 
$$
\lim_{\nu\to\infty}\la\omega^{n-p}\wedge\psi,\del a_\nu\wedge\mu\wedge\psi\ra =0.
$$
Indeed
\begin{equation*}
\begin{split}
\Big|\int_M \langle\omega^{n-p}\wedge\psi,\del a_\nu\wedge\mu\wedge\psi\rangle\vol\Big|&\le
\int_M \vert\omega^{n-p}\vert \vert\mu\vert \vert da_\nu\vert \vert\psi\vert^2 \vol\\
&\le2^{-\nu}\lVert\omega^{n-p}\rVert_{L^\infty}\lVert\mu\rVert_{L^\infty}\lVert\psi\rVert^2_{L^2}.
\end{split}
\end{equation*}
Thus, taking into account \eqref{del*}, we also get that 
$$
\lim_{\nu\to\infty}\la\omega^{n-p}\wedge\psi,a_\nu\del(\mu\wedge\psi)\ra=0.
$$
By the same calculation, we obtain
\begin{equation*}
\begin{split}
0&=\la\delbar^*(\omega^{n-p}\wedge\psi),a_\nu\lambda\wedge\psi\ra\\
&=\la\omega^{n-p}\wedge\psi,\delbar(a_\nu\lambda\wedge\psi)\ra\\
&=\la\omega^{n-p}\wedge\psi,\delbar a_\nu\wedge\lambda\wedge\psi\ra+\la\omega^{n-p}\wedge\psi,a_\nu\delbar(\lambda\wedge\psi)\ra,
\end{split}
\end{equation*}
and both summands approach $0$. Therefore, by the above computations, taking into account \eqref{omegapsi}, we derive
\begin{equation*}
\begin{split}
0=\lim_{\nu\to\infty}&\Big[\la\omega^{n-p}\wedge\psi,a_\nu\del(\mu\wedge\psi)\ra+\la\omega^{n-p}\wedge\psi,a_\nu\delbar(\lambda\wedge\psi)\ra\Big]\\
=\lim_{\nu\to\infty} &\la\omega^{n-p}\wedge\psi,a_\nu\omega^{n-p}\wedge\psi\ra
=\la\omega^{n-p}\wedge\psi,\omega^{n-p}\wedge\psi\ra,
\end{split}
\end{equation*}
by the monotone convergence theorem or the dominated convergence theorem. Thus, $\omega^{n-p}\wedge\psi=0$ and $\psi=0$, since the wedge product by $\omega^{n-p}$, the Lefschetz operator applied $n-p$ times, is an isomorphism at the level of the exterior algebra (see \cite[Prop. 1.2.30]{H}).
\end{proof}

Note that if $M$ is compact in the hypothesis of Theorem \ref{vanishing}, then the proof is even simpler. There is no need of cut-off functions and it suffices to use the Stokes theorem and Lemma \ref{delstar}. Also note that, if $M$ is compact, then the case $n=2$ is no more interesting, since every Gauduchon metric gives rise to a non zero class in Aeppli cohomology. Therefore, as a direct consequence of Theorem \ref{vanishing}, we obtain the following.
\begin{theorem}
\label{cor-cpt}
Let $(M,J,g,\omega)$ be a compact Hermitian manifold of complex dimension $n\ge3$.  Let $1\le p\le n-1$. If $0=[\omega^{n-p}]_{A}\in H^{n-p,n-p}_A(M)$, then $H^{p,0}_{BC}(M)=0$.
\end{theorem}
Finally, as an immediate consequence of Theorem \ref{cor-cpt}, we get:
\begin{theorem}\label{gau}
Let $M$ be a compact complex manifold of complex dimension $n\ge3$. If $M$ carries a Gauduchon metric $\omega$ such that $0=[\omega^{n-1}]_{A}\in H^{n-1,n-1}_A(M)$, then $H^{1,0}_{BC}(M)=0$.
\end{theorem}

\section{Examples and applications}\label{examples}

We start by giving a direct application of Theorem \ref{cor-cpt}.
\begin{example}\label{ex-sl2c}{\rm 
Let $M=\Gamma\backslash G$, where $G$ is the complex Lie group $\SL(2,\C)$, and $\Gamma$ is a discrete cocompact subgroup. As a complex basis for the complex Lie algebra $\g=\sl(2,\C)$ of the Lie group $G$, take the matrices
\begin{equation*}
Z_1=\frac12
\begin{pmatrix}
i&0\\
0&-i
\end{pmatrix},\ \ \ \ 
Z_2=\frac12
\begin{pmatrix}
0&-1\\
1&0
\end{pmatrix},\ \ \ \ 
Z_3=\frac12
\begin{pmatrix}
0&i\\
i&0
\end{pmatrix},
\end{equation*}
such that the structure equations of the Lie algebra are
\begin{equation*}
[Z_1,Z_2]=-Z_3,\ \ \ \ 
[Z_1,Z_3]=Z_2,\ \ \ \ 
[Z_2,Z_3]=-Z_1.
\end{equation*}
Denote by $\phi^1,\phi^2,\phi^3$ the dual basis of $Z_1,Z_2,Z_3$, i.e., $\phi^j(Z_i)=\delta^j_i$. The covectors $\phi^1,\phi^2,\phi^3\in\g^*$ can be seen as a $G$-left-invariant basis of holomorphic $(1,0)$-forms on $G$ and on $M$.
Since $d\alpha(x,y)=-\alpha([x,y])$, for $\alpha\in\g^*$ and $x,y\in\g$, we derive the complex structure equations
\begin{equation}
\label{struc-sl2c}
d\phi^1=\phi^{23},\ \ \ \ d\phi^2=-\phi^{13},\ \ \ \ d\phi^3=\phi^{12}.
\end{equation}
Let us consider on $M$ the Hermitian metric whose fundamental form is
\begin{equation*}
\omega=\frac{i}{2}(\phi^{1\bar{1}}+\phi^{2\bar{2}}+\phi^{3\bar{3}}).
\end{equation*}
Using the structure equations (\ref{struc-sl2c}), one easily checks that $d\omega^2=0$. Moreover, as observed by \cite[p. 467]{OUV}, every left-invariant $(2,2)$-form on $M$ is $d$-exact. It can be seen using again the structure equations (\ref{struc-sl2c}). Hence, $\omega^2$ is exact. Therefore $0=[\omega^{2}]_{A}\in H^{2,2}_A(M)$. Consequently, Theorem \ref{cor-cpt} applies and $H^{1,0}_{BC}(M)=0$. 
To show that $H^{1,0}_{BC}(M)=0$, we can also argue in this way. Let $[\alpha^{1,0}]_{BC}\in H^{1,0}_{BC}(M)$; since $\alpha^{1,0}$ is $d$-closed, then $\delbar\alpha^{1,0}=0$. By \cite[Theorem 2]{W}, the holomorphic cohomology ring of $M$ is isomorphic with the cohomology ring of the complex Lie algebra $\g$ of the complex Lie group $G$. Therefore, by \cite[Theorem 2]{W}, 
$$
\alpha^{1,0} =\sum_{j=1}^3c_j\phi^j,\qquad c_j\in\C,\,\,\, j=1,2,3.
$$
Hence, by the structure equations \eqref{struc-sl2c}, $d\alpha^{1,0}=0$ if and only if $c_j=0,$ for $j=1,2,3$, i.e., $\alpha^{1,0}=0$, so that $H^{1,0}_{BC}(M)=0$.\qed
}
\end{example}

The following example shows that the condition on the Hermitian metric in Theorem \ref{vanishing} is only sufficient and not necessary for the vanishing of $L^2$ closed $(p,0)$-forms in the case of complex dimension $n\ge3$.

\begin{example} {\em Let $M=\mathbb{S}^3\times\mathbb{S}^3$, where $\mathbb{S}^3$ is identified with the special unitary group $SU(2)$ of $2\times2$ matrices. Let $\mathfrak{su}(2)$ be the Lie algebra of $SU(2)$ and denote by $\{E_1,E_2,E_3\}$ and $\{F_1,F_2,F_3\}$ the real basis of the two copies of $\mathfrak{su}(2)$ such that 
$$
\begin{array}{lll}
[E_1,E_2]=2E_3,\,& [E_1,E_3]=-2E_2,\,& [E_2,E_3]=2E_1,\\[5pt]
[F_1,F_2]=2F_3,\,& [F_1,F_3]=-2F_2,\,& [F_2,F_3]=2F_1.
\end{array}
$$
Precisely, using the notation of Example \ref{ex-sl2c}, $E_1=F_1=2Z_3$, $E_2=F_2=2Z_2$, $E_3=F_3=2Z_1$.
Denote by $\{E^1,E^2,E^3\}$, and $\{F^1,F^2,F^3\}$, the dual coframe of $\{E_1,E_2,E_3\}$, respectively $\{F_1,F_2,F_3\}$. Then the following structure equations hold
$$
\left\{
\begin{array}{ll}
 dE^1&=-2E^2\wedge E^3\\
 dE^2&=2E^1\wedge E^3\\
 dE^3&=-2E^1\wedge E^2\\
 dF^1&=-2F^2\wedge F^3\\
 dF^2&=2F^1\wedge F^3\\
 dF^3&=-2F^1\wedge F^2.
\end{array}
\right.
$$
Let $J$ be the almost complex structure on $M$ defined by the following $(1,0)$-complex forms
$$
\left\{
\begin{array}{ll}
 \psi^1&=E^1+iE^2\\
 \psi^2&=F^1+iF^2\\
 \psi^3&=E^3+iF^3.
\end{array}
\right.
$$
Then, the following complex structure equations hold
$$
\left\{
\begin{array}{ll}
 d\psi^1&=i\psi^{13}+i\psi^{1\bar3}\\
 d\psi^2&=\psi^{23}-\psi^{2\bar3}\\
 d\psi^3&=-i\psi^{1\bar1}+\psi^{2\bar2}
 \end{array}
\right.
$$
Therefore $J$ is an integrable almost complex structure on $M$, which turns to be the Calabi-Eckmann 
complex structure on $\mathbb{S}^3\times\mathbb{S}^3$. 
Then, by \cite[p.359]{TT}, the Bott-Chern 
cohomology of $M$ is given by
\begin{equation*}
\begin{split}
H_{BC}^{0,0}(M)&=\span_\C\langle[1]_{BC}\rangle,\\
H_{BC}^{1,1}(M)&=\span_\C\langle[\psi^{1\bar{1}}]_{BC},[\psi^{2\bar{2}}]_{BC}\rangle,\\
H_{BC}^{2,1}(M)&=\span_\C\langle[\psi^{23\bar{2}}+i\psi^{13\bar{1}}]_{BC}\rangle,\\
H_{BC}^{1,2}(M)&=\span_\C\langle[\psi^{2\bar{2}\bar{3}}-i\psi^{1\bar{1}
\bar{3}}]_{BC}\rangle,\\
H_{BC}^{2,2}(M)&=\span_\C\langle[\psi^{12\bar{1}\bar{2}}]_{BC}\rangle,\\
H_{BC}^{3,2}(M)&=\span_\C\langle[\psi^{123\bar{1}\bar{2}}]_{BC}\rangle,\\
H_{BC}^{2,3}(M)&=\span_\C\langle[\psi^{12\bar{1}\bar{2}\bar{3}}]_{BC}\rangle,\\
H_{BC}^{3,3}(M)&=\span_\C\langle[\psi^{123\bar{1}\bar{2}\bar{3}}]_{BC}\rangle,\\
\end{split}
\end{equation*}
where all the representatives are Bott-Chern harmonic with respect to the Hermitian metric $g$ on $M$ 
whose fundamental form is given by 
$$
\gamma =\frac{i}{2}(\psi^{1\bar1}+\psi^{2\bar2}+\psi^{3\bar3})
$$
The other Bott-Chern cohomology groups vanish. In particular, $H^{1,0}_{BC}(M)=0$.
By the above expressions, it follows at once that 
$$
H^{2,2}_A(M)= \span_\C\langle[\psi^{1\bar{1}3\bar3}]_{A},[\psi^{2\bar{2}3\bar3}]_{A} \rangle.
$$
 Let $\omega$ be the 
fundamental form of any Gauduchon metric on $M$. Then, $\omega$ can be expressed in the following way,
\begin{equation}\label{gauduchonM}
2\omega=i(r^2\psi^{1\bar1}+ s^2\psi^{2\bar2}+t^2\psi^{3\bar3})+
u\psi^{1\bar2}-\bar{u}\psi^{2\bar1}+v\psi^{1\bar3}-\bar{v}\psi^{3\bar1}
+w\psi^{2\bar3}-\bar{w}\psi^{3\bar2},
\end{equation}
where $r$, $s$, $t$ are smooth real valued functions on $M$ and $u$, $v$, $w$ are complex valued smooth functions on $M$, satisfying the following conditions
\begin{equation}\label{positivity}
r^2>0,\quad
r^2s^2-\vert u\vert^2>0,\quad
r^2s^2t^2-2\hbox{\rm Re}(iu\bar{v}w)>r^2\vert w\vert^2+s^2\vert v\vert^2+t^2\vert u\vert^2
\end{equation}
A straigthforward computation gives
$$
\begin{array}{lll}
\omega^2&=&-\frac12(r^2s^2\psi^{1\bar12\bar2}+r^2t^2\psi^{1\bar13\bar3}+s^2t^2\psi^{2\bar23\bar3})+\\[5pt]
&{}& +\frac{i}{2}(r^2w\psi^{1\bar12\bar3}-r^2\bar{w}\psi^{1\bar13\bar2}+
s^2v\psi^{2\bar21\bar3}-s^2\bar{v}\psi^{2\bar23\bar1})+\\[5pt]
&{}& +\frac{i}{2}(t^2u\psi^{3\bar31\bar2}-t^2\bar{u}\psi^{3\bar32\bar1}).
\end{array}
$$
Therefore,
$$
\la\omega^2,\psi^{2\bar23\bar3}\ra=\int_M\langle\omega^2,\psi^{2\bar23\bar3}\rangle\vol=
-\frac{1}{2}\int_Ms^2t^2\vol <0.
$$
Since $\psi^{2\bar23\bar3}$ is Aeppli harmonic the last computation shows that $\omega^2$ is not $L^2$-orthogonal to the space of Aeppli harmonic $(2,2)$-forms, so that $0\neq[\omega^2]_{A}\in H^{2,2}_A(M)$.  \qed
}
\end{example}

Reminding that on compact complex surfaces any Gauduchon metric gives rise to a non-zero class in Aeppli cohomology, we give an explicit computation of the Aeppli cohomology class of the fundamental form of Gauduchon metrics being non-zero, while the Bott-Chern cohomology of $(1,0)$-forms is zero, in some specific compact complex surfaces.

\begin{example}\label{ex-sec-kod}{\rm 
Let $M=\Gamma\backslash G$ be a Secondary Kodaira surface. $M$ is a compact complex surface which is diffeomorphic to the solvmanifold $G/\Gamma$, i.e., $G$ is a simply connected solvable Lie group and $\Gamma$ is a closed subgroup of $G$. See \cite{ADT} and \cite{AS} for a more detailed survey of the surface. Denote by $\phi^1,\phi^2$ a $G$-left-invariant coframe of the holomorphic tangent bundle $T^{1,0}M$ with structure equations
\begin{equation}
\label{struc-seck}
d\phi^1=-\frac12\phi^{12}+\frac12\phi^{1\bar{2}},\ \ \ \ d\phi^2=\frac{i}{2}\phi^{1\bar{1}}.
\end{equation}
Following \cite{ADT}, we know that $H^{1,0}_{BC}(M)=0$ and $H^{1,1}_{BC}(M)=\span_\C\langle[\phi^{1\bar{1}}]_{BC}\rangle$. Let $g$ be a Hermitian metric whose fundamental form is
$$
\gamma =\frac{i}{2}(\phi^{1\bar1}+\phi^{2\bar2}).
$$
By equations (\ref{struc-seck}), the $(1,1)$-form $\phi^{1\bar{1}}$ is $BC$-harmonic, then $*\phi^{1\bar{1}}=-\phi^{2\bar{2}}$ is Aeppli-harmonic and $H^{1,1}_{A}(M)=\span_\C\langle[\phi^{2\bar{2}}]_{A}\rangle$. Now, let $\omega$ be the fundamental form of any Gauduchon metric on $M$. In general, $\omega$ can be written as
\begin{equation*}
2\omega=i(A^2\phi^{1\bar{1}}+C^2\phi^{2\bar{2}})+B\phi^{1\bar{2}}-\overline{B}\phi^{2\bar{1}},
\end{equation*}
where $A,B,C$ are \Cinf functions, $A^2>0$, and $A^2C^2-|B|^2>0$. By Hodge decomposition for Aeppli cohomology, $\omega$ can be decomposed as
\begin{equation*}
\omega=E\phi^{2\bar{2}}+\del\mu+\delbar\lambda,
\end{equation*}
where $E\in\C$. If $E\ne0$, then $0\ne[\omega]_{A}\in H^{1,1}_A(M)$, as claimed. By contradiction, assume that $E=0$. Then
\begin{equation*}
\la\omega,\phi^{2\bar{2}}\ra=\la\del\mu+\delbar\lambda,\phi^{2\bar{2}}\ra=0.
\end{equation*}
On the other hand,
\begin{equation*}
\la\omega,\phi^{2\bar{2}}\ra=\frac{i}2\la C^2\phi^{2\bar{2}},\phi^{2\bar{2}}\ra=\frac{i}2\lVert C\rVert_{L^2}^2\ne0.
\end{equation*}
Summing up, $H^{1,0}_{BC}(M)=0$ but $0\ne[\omega]_{A}\in H^{1,1}_A(M)$.\qed
}
\end{example}

Note that the same computations made in Example \ref{ex-sec-kod} for the Secondary Kodaira surface still apply if we consider the Inoue surfaces $\mathcal{S}_M$ and $\mathcal{S}^{\pm}$. See \cite{ADT} and \cite{AS} for a description, for the stucture equations and for the computation of the Bott-Chern cohomology of these surfaces. 

Finally, we give two other applications of the main result. 

Let $M=\Gamma\backslash G$ be a compact nilmanifold of dimension $2n$, that is a compact quotient of a simply connected nilpotent $2n$-dimensional Lie group $G$ by a uniform discrete subgroup $\Gamma$ endowed with a left-invariant complex structure. Let $\mathfrak{g}$ be the Lie algebra of $G$ and denote by $\mathfrak{g}_\C$ the complexification of $\mathfrak{g}$. Then, according to \cite[Theorem 1.3]{Sa}, there exists a basis of left-invariant $(1,0)$-forms $\{\varphi^1,\ldots,\varphi^n\}$ on $M$, indeed 
$\{\varphi^1,\ldots,\varphi^n\}$ is a basis of $\mathfrak{g}_\C^*$, such that 
$$
d\varphi^1=0,\qquad d\varphi^i\in\mathcal{I}(\varphi^1,\ldots,\varphi^{i-1}),\quad i=2,\ldots ,n,
$$
where $\mathcal{I}(\varphi^1,\ldots,\varphi^{i-1})$ is the ideal in $\Lambda^*\mathfrak{g}_\C^*$ spanned 
by $\{\varphi^1,\ldots,\varphi^{i-1}\}$.
Therefore, we immediately get that $\span_\C\langle \varphi^1\rangle \subset H^{1,0}_{BC}(M)$, that is 
$H^{1,0}_{BC}(M)\neq 0$; consequently we derive the following.
\begin{proposition}\label{prop1}
Let $M=\Gamma\backslash G$ be a compact nilmanifold endowed with a left-invariant complex structure. Then every Gauduchon metric $\omega$ on $M$ is such that $0\ne[\omega^{n-1,n-1}]_{A}\in H^{n-1,n-1}_A(M)$.
\end{proposition}
Let $(M,J,g,\omega)$ be a Hermitian manifold of complex dimension $n$. The Hermitian metric 
$g$ is said to be {\em Strong K\"ahler with Torsion} or, shortly, {\em SKT}, if
$\partial \overline \partial \omega =0$ or, equivalently, $d d^c \omega =0$.
Hermitian SKT metrics have been studied by many
authors and they have also applications in type II string theory
and in $2$-dimensional supersymmetric $\sigma$-models
\cite{GHR,Str,IP}. Moreover, they have also relations with generalized
K\"ahler geometry (see for instance
\cite{GHR,Gu,Hi2}).
In the terminology by Streets and Tian, SKT metrics are called {\em pluriclosed metrics}. In \cite{ST}, a curvature evolution equation on compact complex manifolds endowed with pluriclosed metrics is introduced and studied.

From \cite[Theorem 1.2]{FPS}, we deduce the following application of our main result.
\begin{proposition}\label{prop2}
Let $M = \Gamma\backslash G$ be a nilmanifold of real dimension $6$ with an invariant complex structure $J$. Then every $SKT$ metric $\omega$ on $M$ is such that $0\ne[\omega]_{A}\in H^{1,1}_A(M)$.
\end{proposition}
\begin{proof}
By \cite{FPS}, the $SKT$ condition is satisfied by either all invariant Hermitian metrics $\omega$ on $M$ or by none. Indeed, it is satisfied if and only if $J$ has a basis $(\alpha^i)$ of $(1,0)$-forms such that
\begin{equation}
\label{skt-nil}
\begin{split}
d\alpha^1=0,\ \ \ 
d\alpha^2=0,\ \ \ 
d\alpha^3=A\alpha^{\bar12}+B\alpha^{\bar22}+C\alpha^{1\bar1}+D\alpha^{1\bar2}+E\alpha^{12},
\end{split}
\end{equation}
where $A,B,C,D,E$ are complex numbers such that
\begin{equation*}
|A|^2+|D|^2+|E|^2+2\real(\overline{B}C)=0.
\end{equation*}
Therefore, if $\omega$ is a $SKT$ metric on $M$, then by structure equations (\ref{skt-nil}) we get $\span_\C\langle \alpha^1\wedge\alpha^2\rangle \subset H^{2,0}_{BC}(M)$, that is $H^{2,0}_{BC}(M)\ne0$. Thus by Theorem \ref{cor-cpt} we derive $0\ne[\omega]_{A}\in H^{1,1}_A(M)$.
\end{proof}

\end{document}